\documentclass{amsart}

\usepackage{color,graphicx}
\usepackage{amsmath,latexsym,amssymb,amsfonts,amsbsy, amsthm}
\usepackage{verbatim,tikz,tikz-cd}
\usepackage{bbm}
\usepackage{dsfont}
\usepackage{amssymb}
\usepackage{rotating} 

\newtheorem{theorem}{Theorem}[section]
\newtheorem{lemma}[theorem]{Lemma}
\newtheorem{proposition}[theorem]{Proposition}

\newtheorem{assumption}[theorem]{Assumption}
\newtheorem{conjecture}[theorem]{Conjecture}

\theoremstyle{definition}
\newtheorem{definition}[theorem]{Definition}
\newtheorem{example}[theorem]{Example}

\theoremstyle{remark}
\newtheorem{remark}[theorem]{Remark}

\newcommand{\be}{\begin{equation}}
\newcommand{\ee}{\end{equation}}
\newcommand{\bee}{\begin{equation*}}
\newcommand{\eee}{\end{equation*}}
\newcommand{\ba}{\begin{aligned}}
\newcommand{\ea}{\end{aligned}}
\numberwithin{equation}{section}

\begin{document}

\title{Coisotropic branes on tori and Homological mirror symmetry}

\author{Yingdi Qin}
\address{Department of Mathematics, University Of Penn, Philadelphia, PA 19103}
\email{yqin@sas.upenn.edu}


\date{July 1, 2020 and, in revised form, May 22, 2022.}



\begin{abstract}
Homological mirror symmetry (HMS) asserts that the Fukaya category of a symplectic manifold is derived equivalent to the category of coherent sheaves on the mirror complex manifold. Without suitable enlargement (split closure) of the Fukaya category, certain objects of it are missing to prevent HMS from being true. One possible solution is to include coisotropic branes into the Fukaya category.
This paper gives a construction for linear symplectic tori of a version of Fukaya category including coisotropic branes by using a doubling procedure, and discussing the relation between the Fukaya category of the doubling torus and the Fukaya category of the original torus.
\end{abstract}

\maketitle

\tableofcontents
%

\section{Introduction}

Mirror symmetry is a fascinating relation between a pair of Calabi-Yau manifolds originating from physics. Even though two mirror manifolds may look very different geometrically, they are predicted to give rise to equivalent quantum field theories and thus equivalent notions of physics.\\

It took mathematicians decades to understand this relation. For a pair of mirror Calabi-Yau manifolds $X$ and $X^\vee$, the first thing to notice is that their Hodge diamonds are reflections of each other, i.e. $h^{p,q}(X)=h^{q,\,p}(X^\vee)$. In 1990, physicists Philip Candelas, Xenia de la Ossa, Paul Green, and Linda Parkes showed that the enumerative invariants and period integrals of a mirror pair were related to each other, which led to the calculations of Gromov-Witten invariants on quintic 3-folds. \\

In 1994, Maxim Kontsevich proposed the celebrated Homological Mirror Symmetry conjecture which provides a mathematically rigorous explanation for the mysterious mirror symmetry phenomenon. It asserts that the Fukaya category of a symplectic manifold is derived equivalent to the category of coherent sheaves of its mirror dual (usually a complex manifold). \\

On the other hand, Strominger-Yau-Zaslow (SYZ)'s conjecture\cite{SYZ} gives a geometric intuition for mirror symmetry. It conjectures that mirror symmetry is a torus duality between symplectic geometry (A-side) and complex geometry (B-side), which can be
interchanged by implementing a Fourier-transform on torus fibers. Within this geometric framework (following Fukaya\cite{F} and Abouzaid\cite{AbF}), one could construct a mirror functor from the Fukaya category of the A-side to the category of coherent sheaves of the B-side which induces the derived equivalence.
\subsection{Homological mirror symmetry and coisotropic branes}
The Fukaya category consists of Lagrangian branes (Lagrangian submanifolds equipped with flat connections) as objects and Lagrangian intersections as morphisms, and the Floer products (compositions of morphisms in the Fukaya category) are determined by counting holomorphic triangles with vertices at Lagrangian intersections and with edges in Lagrangian submanifolds.
In general, Lagrangian submanifolds do not generate the version of Fukaya category for which homological mirror symmetry holds as stated, certain objects are missing. For example, Let $E_\tau$ be a elliptic curve with a Teichmuller parameter $\tau$. Suppose $E_\tau$ admits a complex multiplication, for example, $\tau =i$, then multiplication by $i$ is an automorphism of $E_\tau$ which is called a complex multiplication. Let $E=E_\tau^2$ be the product abelian surface, then the mirror of $E$ is a 4 dimensional symplectic torus $(T,\,\omega)$. Considering the Grothendieck groups of the categories involved in homological mirror symmetry and their images under the Chern character map, we expect a commutative diagram.
$$
\begin{tikzcd}
K_0(D^bCoh(E)) \arrow[rr] \arrow[d, "Ch"] &  & K_0(Fuk(T,\omega)) \arrow[d, "Ch"] \\
H^*(E) \arrow[rr]                         &  & H^*(T).
\end{tikzcd}
$$
By comparing the images of the Chern character maps, one finds that $Im(Ch)\otimes \mathbb{Q}$ for $K_0(D^bCoh(E))$ is 6 dimensional, yet $Im(Ch)\otimes \mathbb{Q}$ for $K_0(Fuk(T))$ is 5 dimensional, taking values in the kernel of 
$$  \omega\wedge:H^2(T;\mathbb{Q})\rightarrow H^4(T;\mathbb{Q}),$$
thus the two categories cannot be equivalent. This indicates that Lagrangians do not generate a sufficiently large category. Kapustin and Orlov \cite{KO} suggest that the missing objects are coisotropic submanifolds equiped with a $U(1)$ connection satisfying certain conditions.
The difficulty of incorporating coisotropic submanifolds is to define the morphisms involving coisotropic objects and the product operations involving these morphisms. It is a long standing problem to define the appropriate Fukaya category including coisotropic branes proposed by Kapustin-Orlov \cite{KO} and further studied by Aldi-Zaslow \cite{AZ}, Chan-Leung-Zhang \cite{CLZ}, Herbst \cite{H}.

\begin{remark}
	An alternative way to enlarge the Fukaya category is to take the split closure of the derived Fukaya category, i.e. enlarging the Fukaya category by adding formal direct summands of objects\cite{S} \cite{AS}.
\end{remark}

\subsection{Main result}

On a linear symplectic torus $(T,\,\omega)$, a new method is proposed to extend the Fukaya category of the torus to include coisotropic submanifolds alongside linear Lagrangian submanifolds as objects of the category. The approach is by considering a twisted doubling torus $T\times T^\vee$ (see definition \ref{Double}) of $T$ and lifting (possibly coisotropic) objects into Lagrangians of $T\times T^\vee$.
\begin{theorem}\label{LiftisLag}\quad
	\begin{enumerate}
		\item The lift of a (possibly coisotropic) linear object on $T$ is a Lagrangian submanifold of $T\times T^\vee$.
		\item The lift is a complex submanifold with respect to a canonical complex structure on $T\times T^\vee$.
	\end{enumerate}

\end{theorem}
The Lagrangian Floer theory of $T$ is naturally related to the Lagrangian Floer theory of the twisted doubling torus $T\times T^\vee$ of $T$. However, the morphism spaces that we want to consider are only a certain subspace of the Floer cohomology in doubling torus, which we call the ``u-part".
The main result, informally, is that the Floer cohomology of two objects in $T$ is isomorphic to the ``u-part" Floer cohomology of the lifts in $T\times T^\vee$ and respects the Floer products.
\begin{theorem}
	For a pair of Lagrangian branes $L,\,L'$ which are mirror to a pair of line bundles $\mathcal{L},\,\mathcal{L'}$, let $\boldsymbol{L},\,\boldsymbol{L}'$ be the lifts in double torus. Suppose $\mathcal{L'}\otimes \mathcal{L}^{-1}$ is ample. Then the ``u-part" Floer cohomology $HF_u^*(\boldsymbol{L},\boldsymbol{L}')$ is isomorphic to $HF^*(L,\,L')$. And for two such pair $L,L'$ and $L',L''$, the following diagram commutes
	$$\begin{tikzcd}
	{HF^*(L',\,L'')\otimes HF^*(L,\,L')} \arrow[rr] \arrow[d, "\cong"] &  & {HF^*(L,\,L'')} \arrow[d, "\cong"] \\
	HF^*_u(\boldsymbol{L}',\,\boldsymbol{L}'')\otimes HF^*_u(\boldsymbol{L},\,\boldsymbol{L}') \arrow[rr]                                 &  & HF_u(\boldsymbol{L},\boldsymbol{L}'').
	\end{tikzcd}$$
\end{theorem}
 By Theorem \ref{LiftisLag}, a coisotropic submanifold of $T$ lifts to a Lagrangian submanifold of the doubling torus which could be studied using Lagrangian Floer theory. Thus, the enlarged Fukaya category including coisotropic branes is realized as a non-full subcategory of the Fukaya category of the doubling torus. As a corollary of this doubling construction, the Fukaya category of $T$ is equivalent to the Fukaya category of the dual torus $T^\vee$.

\section{HMS for tori}

Homological mirror symmetry for tori was extensively studied by Polishchuk-Zaslow\cite{PZ}, Fukaya\cite{F}, Kontsevich and Soibelman\cite{KS}. Here we will review the construction of the mirror manifold of a symplectic torus, and how Lagrangian branes correspond to coherent sheaves on the mirror.

\subsection{SYZ fibrations and construction of mirrors}

$T$ be a torus equipped with a complex valued closed 2-form $B+i\omega$, where real part is called B-field, and where imaginary part is a (non-degenerate) symplectic form.
\begin{definition}
	A SYZ fibration for $(T,\,B+i\omega)$ is a torus fibration such that each fiber is a Lagrangian submanifold.
\end{definition}

Given a SYZ fibration of $(T,\,B+i\omega)$ with base $Q$ and fibers $F_q$, $T_qQ$ is naturally identified with $H^1(F_q;\, \mathbb{R})$ by $v\mapsto [\iota_v\omega]$, where we lift $v\in T_qQ$ to a normal vector field along $F_q$, also denoted $v$. Let $T^{\mathbb{Z}}Q\subset TQ$ be the lattice corresponding to $H^1(F_q,\, \mathbb{Z})$, and $T^*_\mathbb{Z}Q\subset T^*Q$ the dual lattice, which is natrually isomorphic to $H_1(F;\,\mathbb{Z})$. The key property of this lattice $T^*_\mathbb{Z}Q$ is that its sections are locally exact 1-forms on $Q$. Indeed, for a class $\beta\in H_1(F,\,\mathbb{Z})$, near a point $q_0\in Q$, let $x_\beta(q)=\int_{(q-q_0)\times \beta}\omega$, then $dx_\beta(v)=\int_{\beta}\iota_v\omega$ is the corresponding section of $T_\mathbb{Z}^*Q$. If $\beta_1,...,\,\beta_n$ form a basis of $H_1(F;\,\mathbb{Z})$, then the local coordinates $x_{\beta_1},...,x_{\beta_n}:\,Q\rightarrow \mathbb{R}^n$ induce an integral affine structure on $Q$, and this structure doesn't depend on the choice of basis of $H_1(F;\,\mathbb{Z})$.
\begin{assumption}
	The restriction of the B-field $B$ to the SYZ fibers vanishes.
\end{assumption}
This assumption eliminates the case with noncommutative mirrors, and allows one to construct the mirror manifold $Y$ of $(T,\,B+i\omega)$ as the moduli space of SYZ fibers equipped with flat $U(1)$ connections. $Y$ naturally admits a complex structure, specified by local coordinates
\begin{equation}\label{coord}
z_\beta(F_q,\,\nabla)=e^{2\pi i \int_{(q-q_0)\times \beta}(B+i\omega)}hol_{\nabla}(\beta)
\end{equation}
where $F_{q_0}$ is a fixed fiber, $\nabla$ is a flat $U(1)$ connection on $F_q$, and $(q-q_0)\times \beta $ is a relative homology class in $H_2(T,\,F_q\cup F_{q_0};\,\mathbb{Z})$ with boundary $\beta$ in $H_1(F_q;\,\mathbb{Z})$ and $-\beta$ in $H_1(F_{q_0};\,\mathbb{Z})$.
The tangent space of $Y$ at $(F,\,\nabla)$ is the quotient of the set of all pairs $(v,\,\alpha)\in C^{\infty}(NF)\oplus \Omega^1(F,\,\mathbb{R})$ such that $v$ is an infinitesimal Lagrangian deformation, and $\alpha$ is a closed $1$-form, viewed as an infinitesimal deformation of the flat connection, by the subspace consisting of Hamiltonian vector fields and  exact 1-forms (which correspond to trivial deformations).

\begin{lemma}
	$T_{(F_q,\,\nabla)}Y$ is identified with $H^1(F_q,\,\mathbb{C})$ via the map
	\begin{equation}
	\phi :(v,\,\alpha)\mapsto \iota_v (\omega-iB)+i\alpha.
	\end{equation}
	And the map is complex linear.
\end{lemma} 

\begin{proof}
	The Arnold-Liouville theorem implies that there are canonical identifications 
	\begin{equation}
	T_qQ \cong H^1(F_q;\,\mathbb{R}),\, v\mapsto [\iota_v \omega]
	\end{equation}
	where $v$ is lifted from $T_qQ$ to $C^{\infty}(NF)$, so $v\mapsto [\iota_v \omega]$ maps bijectively to $H^1(F_q;\,\mathbb{R})$. And an exact $1$-form $\alpha$ is simply a gauge transformation which do not contribute to the deformation of connections, so the deformations of connections are classified by cohomology class of the 1-form $\alpha$. So $\phi$ is a bijection.
	To verify $\phi$ is complex linear, we see that
	\begin{equation}
	dlog(z_\beta)(v,\,\alpha)=-2\pi  \int_{\beta}(\iota_v (\omega-iB)+i\alpha)
	\end{equation}
	is complex linear for every holomorphic coordinate function $z_\beta$. So $\phi$ is complex linear. 
\end{proof}

In fact, using the coordinates from \eqref{coord} and observing that interior product with $B+i\omega$ defines an (injective) linear map $H_1(Q) \rightarrow H^1(F;\,\mathbb{C})$, we can identify the mirror manifold $Y$ with
\begin{equation}
H^1(F;\,\mathbb{C})/H^1(F;\,\mathbb{Z})+(B+i\omega) H_1(Q;\,\mathbb{Z}).
\end{equation} 
\begin{example}\label{stdex}
	For $T=(\mathbb{R}/\mathbb{Z})^{2n}$, $B+i\omega=\tau dr\wedge d\theta=\sum \tau_{jk}dr_j\wedge d\theta_k (\tau \in M_{n\times n}(\mathbb{C}))$, with SYZ fibers $F=\{r\}\times T_\theta ^n$, the mirror complex torus is $E=\mathbb{C}^n/(\mathbb{Z}^n+\tau^{T} (\mathbb{Z}^n))$.
\end{example}

\subsection{Floer cohomology and sheaf cohomology}

\begin{figure}
	\centering 
	\includegraphics[scale=0.8]{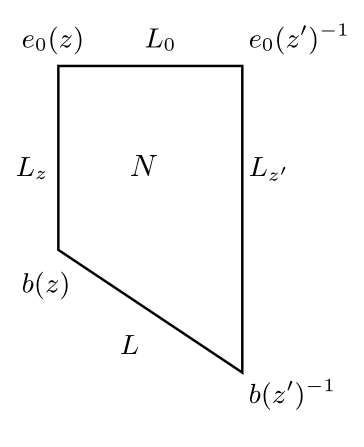}
	\caption{The trapezoid bound by $L_0$, $L_z$, $L$ and $L_{z'}$. }\label{holo}
\end{figure}

\begin{definition}
	A linear Lagrangian brane in $(T,\, B+i\omega)$ is a linear Lagrangian submanifold $L\subset (T,\, \omega)$ (with a choice of grading and spin structure) together with a complex line bundle over $L$ with a unitary connection $\nabla$ whose curvature satisfies $F=-B|_L$.
\end{definition}
\begin{assumption}
	Assume there exists a linear Lagrangian brane $L_0$ with trivial local system such that $L_0\cap F_q$ at only 1 point for each $q\in Q$.
\end{assumption}
This $L_0$ will serve as the mirror of structure sheaf on $Y$. In the case of Example \ref{stdex}, we simply choose $L_0=\{\theta=0\}$. The points of the mirror torus parametrize Lagrangian branes $L_z=(F_q,\, \nabla)$ supported on SYZ fibers, and the generators $e_0(z)\in CF(L_0,\, L_z)$ are chosen to serve as evaluation map of the structure sheaf at $z$.
Given a linear Lagrangian brane $(L,\, \nabla)$ in $(T,\, B+i\omega)$, assume $L$ is transversal to $F_q$ for each $q\in Q$. Now we construct its mirror sheaf to be a vector bundle
\bee
\mathcal{L}=\cup_{z\in Y}CF^*(L_z,\, L)\rightarrow Y.
\eee
When we move $L_z$, let $x(z)$ be a local continuous section of the intersection $L_z\cap L$, and let $b(z)\in CF(L_z,\, L)$ be a rescaling of $x(z)\in L_z\cap L$. Then $b(z)$ is a local section of $\mathcal{L}$.
\begin{definition}
	$\mathcal{L}$ admits a natural holomorphic structure such that, if locally for any $z,z'\in Y$, and N is a trapezoid bound by $L_0$, $L_z$, $L$, $L_{z'}$, see figure \ref{holo},
	\bee
	e^{2\pi i\int_{N}B+i\omega}hol(\partial N)\equiv 1,
	\eee
	then $b(z)$ is a holomorphic section.
	Here $hol(\partial N)$ is the composition of $e_0$, parallel transport along $L_z$, $b(z)$,parallel transport along $L$, $b(z')^{-1}$, parallel transport along $L_{z'}$, $e_0(z')^{-1}$ and parallel transport along $L_0$. 
\end{definition}

\begin{remark}
	These holomorphic sections uniquely determine the holomorphic structure of $\mathcal{L}$. Here we only consider linear Lagrangians, in which case the Floer differentials are automatically zero, so the corresponding mirror sheaves are holomorphic vector bundles. In general, with Floer differential, one obtains chain complexes of locally free sheaves \cite{F} \cite{AbF}. 
\end{remark}

Under homological mirror symmetry, we have 

\begin{theorem}\cite{F} \cite{KS}
	\bee
	HF^*(L_1,L_2)\cong Ext^*(\mathcal{L}_1,\mathcal{L}_2)
	\eee
	And the following diagram commutes:  \\
	
	\bee
	\begin{tikzcd}
	{HF^*(L_2,\,L_3)\otimes HF^*(L_1,\,L_2)} \arrow[rr, "\mu_2"] \arrow[d, "\cong"]                 &  & {HF^*(L_1,L_3)} \arrow[d, "\cong"]    \\
	{Ext^*(\mathcal{L}_2,\,\mathcal{L}_3)\otimes Ext^*(\mathcal{L}_1,\,\mathcal{L}_2)} \arrow[rr] &  & {Ext^*(\mathcal{L}_1,\,\mathcal{L}_3)}
	\end{tikzcd}
	\eee

\end{theorem}

The above theorem is best illustrated by the example of theta functions on the elliptic curve.
\begin{figure}	\label{line1}
	\centering 
	\includegraphics[scale=1]{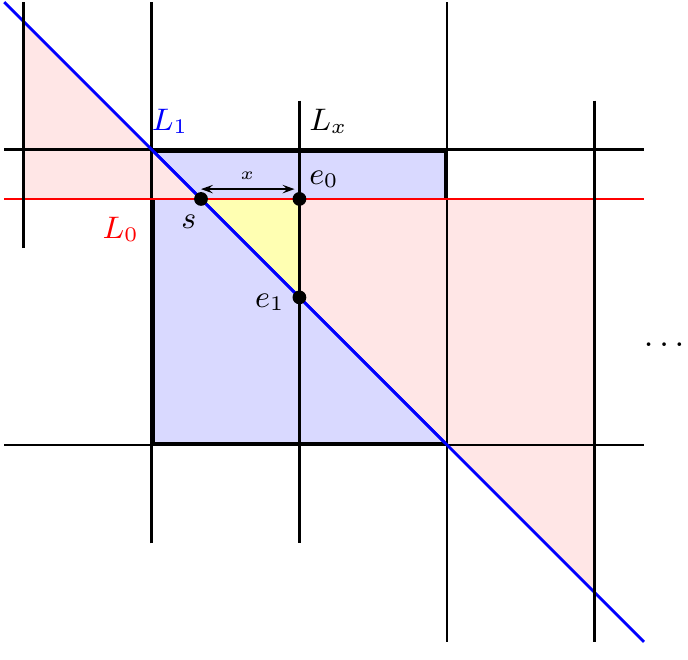}
		\caption{Holomorphic triangles computing the Floer product $\mu^2(e_1,s)\in HF^*(L_0,L_x)$. }\label{holo1}
\end{figure}

\begin{example}
	Let $T=\mathbb{R}^{2}/\mathbb{Z}^{2}$, $B+i\omega=\tau dr\wedge d\theta$ be a symplectic two torus with coordinates $r,\,\theta$. Let $L_0$ be a horizontal Lagrangian $\{\theta=0\}$ (mirror to the structure sheaf), $L_d$ be a slope $-d$ Lagrangian $\{\theta=-dr\}$ (mirror to a degree d line bundle $\mathcal{L}^d$), $L_z$ be a vertical Lagrangian with position $x$ and a connection with holonomy $e^{-2\pi iy}$, $(\{r=x\},\,\nabla=d+2\pi iyd\theta)$, where $z=x+iy$. See figure \ref{holo1} for the case $d=1$.
	The generators $s_k=(\frac{k}{d},\,0)\in L_0\cap L_d$ of 
	$HF(L_0,\,L_d)$ correspond to the $\vartheta$-basis of $H^0(E,\,\mathcal{L}^d)$
	
	$$\vartheta_{d,\,k} =\sum_{n\in \mathbb{Z}} e^{\pi\tau id (n-\frac{k}{d})^2}e^{2\pi id(n-\frac{k}{d}) z}.$$
\end{example}

\subsection{Coisotropic branes}

Kapustin and Orlov introduce the following notion of coisotropic brane, motivated by a string theoretical calculation.

\begin{definition}
	Given a symplectic manifold $(X^{2n},\,B+i\omega)$, a coisotropic brane is a coisotropic submanifold $C^{n+k}$ equiped with a complex line bundle $(\mathcal{L},\,\nabla)$ such that
	\begin{enumerate}
		\item   Let $-2\pi iF$ be the curvature of $(\mathcal{L},\,\nabla)$, then $F+B|_C$ vanishes on $TC_{iso}= ker\ \omega|_C$, where $F+B|_C$ is viewed as a bundle morphism $TC\rightarrow TC^*$. In particular, $F=-B|_C$ along the isotropic leaves (foliated by $ker\, \omega|_C=TC_{iso}$).
		\item  $\omega^{-1}(F+B|_C)$ defines a transverse almost complex structure on $C$, i.e. an almost complex structure on $TC_{red}=TC/TC_{iso}$. Equivalently, $\omega+(F+B)\omega^{-1}(F+B)=0$ on $TC$
	\end{enumerate}
	
\end{definition}
\begin{remark}
	The second condition implies that $F+B+i\omega$ is a holomorphic symplectic form on the space of isotropic leaves, hence forces $k$ to be even. In the case $n=2,\,k=2$ and $B=0$, we have a space filling coisotropic brane, and this condition is equivalent to $$\omega\wedge F=0,\, \omega\wedge\omega=F\wedge F.$$
	Since $F\wedge F$ represents an integral cohomology class, coisotropic branes can only arise for some special $\omega$.
\end{remark}
\begin{remark}
	The transverse almost complex structure arising from the geometry of the coisotropic brane is always integrable \cite{KO}.
\end{remark}

\begin{example}
	Let $(T=\mathbf{R}^4/\mathbf{Z}^4,\, \omega=dr_1\wedge d\theta_1 + dr_2\wedge d\theta_2)$ be the standard symplectic four torus. 
	Then $(C=T,\,\nabla=d+2\pi ir_1d\theta_2-2\pi ir_2d\theta_1)$ is a coisotropic brane. And the induced complex stucture has complex coordinates $r_1-ir_2$, $\theta_1+i\theta_2$.
\end{example}

\noindent
Kapustin and Orlov have made a proposal for the endomorphisms of a coisotropic brane, namely $End(C)\simeq H^{0,\,*}(C)$, where Dolbeault cohomology is considered with respect to the transverse complex structure, but until now it was not understood how to define morphisms between different branes. In the next sections we propose a definition based on a ``doubling" construction.

\section{Doubling and lifting}

\subsection{Construction of the twisted double torus and lift of coisotropic branes   }
\subsubsection{Construction for symplectic torus without B-field}

Let $(T=V/\Lambda, \,\omega)$ be a linear symplectic torus. Let $(T^\vee=V^\vee /\Lambda^\vee, \,-\omega^{-1})$ be the dual torus with the inverse symplectic form $-\omega^{-1}(\alpha, \,\beta):=\alpha(\omega^{-1}\beta)$.

We introduce the following doubling procedure:

\begin{definition}\label{Double}
	The twisted doubling torus of $(T, \,\omega)$ is a symplectic torus with a B-field
$$(T\times T^\vee, \  \frac{1}{2}\omega\oplus-\frac{1}{2}\omega^{-1}, \ \sigma_0=\sum \frac{1}{2} dx_j\wedge d\hat{x}_j)$$
	where $x_j$ are coordinates on $T$ and $\hat{x_j}$ are the dual coordinates on $T^\vee$. $\sigma_0=\sum \frac{1}{2} dx_j\wedge d\hat{x}_j$ is the B-field which does not depend on the choice of coordinates.
\end{definition}

\begin{remark}\label{AS}
	The twisted doubling torus naturally comes with a complex structure $J$ which sends a tangent vector ${v}$ in $T$ to its symplectic dual $v \lrcorner \omega$ as a tangent vector of $T^\vee$. In matrix notation,
	
	$$ J=\left( \begin{array}{ccc}
	0 & \omega^{-1} \\
	-\omega & 0
	\end{array}\right).$$
	
\end{remark}

\begin{example}
	Let $T=\mathbb{R}^4/\mathbb{Z}^4$, $\omega=dr_1\wedge d\theta_1+dr_2\wedge d\theta_2$. Its dual torus is $T^\vee=\mathbb{R}^4/\mathbb{Z}^4$, $-\omega^{-1}=d\hat{r}_1\wedge d\hat{\theta}_1+d\hat{r}_2\wedge d\hat{\theta}_2$. The twisted double torus is $T\times T^\vee=\mathbb{R}^4/\mathbb{Z}^4\times \mathbb{R}^4/\mathbb{Z}^4$, $ \frac{1}{2}\omega\oplus-\frac{1}{2}\omega^{-1}=\frac{1}{2}(dr_1\wedge d\theta_1+dr_2\wedge d\theta_2+d\hat{r}_1\wedge d\hat{\theta}_1+d\hat{r}_2\wedge d\hat{\theta}_2)$ and $B$ as above.
	
\end{example}

One can lift a linear Lagrangian brane or a coisotropic brane of $(T, \,\omega)$ to a Lagrangian brane in the doubling torus.\\

For a Lagrangian $(L, \,\nabla)$, its lift is $\boldsymbol{L}=(L\times L^\perp, \,\nabla\otimes \mathds{1})$, where $L^\perp$ is the conormal of $L$ translated by the holonomy of $\nabla$.\\

For a coisotropic brane $(C, \,\nabla)$, its lift $\boldsymbol{C}$ is a graph over $T$ in the doubling torus determined by the holonomy of $\nabla$.

\begin{definition}
	The lift of a coisotropic (posibly Lagrangian) brane $(C, \,\nabla)$ is defined to be
	$$\{(x, \,\hat{x})\in T\times T^{\vee}|x\in C\  \text{and} \
	\langle\hat{x}, \,\gamma_x\rangle =(-1)^{\xi(\gamma_x)}hol_\nabla(\gamma_x),\ \forall\ \gamma_x \in \pi_1(C,x)\},\quad \pi_T^*\nabla$$
	where $\gamma_x$ is any linear circle passing through $x$ within $C$. And $\xi:\,H_1(C)\rightarrow \mathbb{Z}/2$ is such that
	\bee
	\xi(\gamma+\gamma')-\xi(\gamma)-\xi(\gamma')=c_1(\nabla)(\gamma\wedge \gamma') \ mod \ 2
	\eee
\end{definition}

\begin{remark}
	$\xi$ is introduced to make sure $(-1)^{\xi(\cdot)}hol_\nabla(\cdot)$ is a homomorphism: $\pi_1(C, \,x)\rightarrow U(1)$.
	$\xi$ has a similar role to the spin structure.
	The space of different choices of $\xi$ is an affine space over $H^1(C; \,\mathbb{Z}/2)$.
\end{remark}

\begin{proposition}\label{LiftisLag1}
	The lift of a coisotropic (possibly Lagrangian) brane is a Lagrangian brane in the doubling torus. And it is also a $J$-complex (see remark \ref{AS}) submanifold of $T\times T^\vee$.
\end{proposition}
\begin{proof}
	Let $(C, \,\nabla)$ be a coisotropic brane, $\boldsymbol{C}$ be its lift. By linearizing the definition of the lift, we have
	$$T\boldsymbol{C}=\{u+Fu+\omega v| u\in TC, v \in TC_{iso}\}
	=\left(\begin{array}{ll}
	1\ &0 \\
	F\ &1
	\end{array}\right)
	\left(\begin{array}{ll}
	& TC \\
	& \omega\ TC_{iso}
	\end{array}\right), $$
	where $F$ denote the curvature of $\nabla$.
	Given two tangent vectors $u_1+Fu_1+\omega v_1$, $u_2+Fu_2+\omega v_2$ of $T\boldsymbol{C}$,
	\begin{equation*}
	\begin{aligned}
	&\omega\oplus-\omega^{-1}(u_1+Fu_1+\omega v_1,\ u_2+Fu_2+\omega v_2) \\
	=&\omega(u_1,\,u_2)-\omega^{-1}(Fu_1+\omega v_1,\,Fu_2+\omega v_2) \\	=&\omega(u_1,\,u_2)-\omega^{-1}(Fu_1,\,Fu_2)-\omega^{-1}(\omega v_1,\,Fu_2)-\omega^{-1}(Fu_1,\,\omega v_2)-\omega^{-1}(\omega v_1,\,\omega v_2)\\
	=&(\omega+F\omega^{-1}F)(u_1,\,u_2)+F(v_1,\,u_2)+F(u_1,\,v_2)+\omega(v_1,\,v_2)\\
	=&0
	\end{aligned}
	\end{equation*}
	The last line equals $0$ because $\omega+F\omega^{-1}F=0$ and $F$ is zero on $TC_{iso}$. So the lift $\boldsymbol{C}$ is Lagrangian. \\
		
	To prove that $\boldsymbol{C}$ is a $J$-complex submanifold of $T\times T^\vee$,
	we observe that, for $u \in TC$,
	$\omega u +F\omega^{-1}Fu $ vanishes on $TC$,
	hence equals $\omega v'$ for some $v'\in TC_{iso}$. Therefore,
	\begin{equation*}
	\begin{aligned}
	&J(u+Fu+\omega v)=\omega^{-1}Fu+v-\omega u \\
	=&\omega^{-1}Fu+v+F\omega^{-1}F u -\omega v'\\
	=&(\omega^{-1}Fu+v)+F(\omega^{-1}Fu+v) -\omega v' \in T\boldsymbol{C}.
	\end{aligned}
	\end{equation*}
	
\end{proof}

\begin{remark}
	There is an ambiguity when writing $Fu$ as a vector of $TT^\vee$, in fact $Fu\in T^*C$ whose lift to $T^*T$ is only unique up to a vector in $\omega TC_{iso}$. In the above calculation, a lift of $Fu$ to $T^*T$ is chosen. Another way is to write $T\boldsymbol{C}=\{u+f|<f,\,v>=F(u,\,v)\ \forall v\in TC\}$.
\end{remark}

\subsubsection{Construction for symplectic torus with B-field}
Now we deal with the more general situation when the symplectic torus starts with a B-field. Let $(T=V/\Lambda,\, \omega,\, B)$ be a linear symplectic torus equipped with a B-field $B\in H^2(T;\,\mathbb{R})$.
\begin{assumption}\label{assump}
	$Id+(\omega^{-1}B)^2$ is invertible.
\end{assumption}
The dual torus is  $(T^\vee=V^\vee /\Lambda^\vee,\ -(\omega+B\omega^{-1}B)^{-1},\  (\omega+B\omega^{-1}B)^{-1}B\omega^{-1})$.

\begin{remark}
	These formular for the symplectic form and B-field on the dual torus are rather confusing at first sight. In fact, they are the imaginary part and real part of $(B+iw)^{-1}$. Indeed,
\bee\ba
(B+i\omega)^{-1}&=-i(Id-i\omega^{-1}B)^{-1}\omega^{-1}=-i\sum_{k\geq 0} (i\omega^{-1} B)^k \omega^{-1}\\
&=-i(\omega+B\omega^{-1}B)^{-1}+(\omega+B\omega^{-1}B)^{-1}B\omega^{-1}. \ea\eee
\end{remark}

\begin{definition}
	The twisted doubling torus of $(T,\,\omega,\, B)$ is a symplectic torus with a B-field
	\begin{equation*}
	(T\times T^\vee,\quad  \frac{1}{2}
	\left(\begin{array}{cc}
	 \omega+B\omega^{-1}B & B\omega^{-1}\\
		 -\omega^{-1}B & -\omega^{-1}
		\end{array}\right),
		\quad
		 \sigma_0=\sum_j \frac{1}{2} dx_j\wedge d\hat{x}_j)	
	\end{equation*}
	where $x_j$ are coordinates on $T$ and $\hat{x}_j$ are the dual coordinates on $T^\vee$. $\sigma_0=\sum \frac{1}{2} dx_j\wedge d\hat{x}_j$ is the B-field which does not depend on the choice of coordinates.
\end{definition}

\begin{remark}
	If we start with the dual torus  $(T^\vee=V^\vee /\Lambda^\vee,\,-(\omega+B\omega^{-1}B)^{-1},\, (\omega+B\omega^{-1}B)^{-1}B\omega^{-1})$, then the twisted double torus of the dual torus is
		\bee
		(T^\vee\times T,\quad \frac{1}{2}
		\left(\begin{array}{cc}
		-\omega^{-1} & -\omega^{-1}B\\
		 B\omega^{-1} & \omega+B\omega^{-1}B
		\end{array}\right),
		\quad
		-\sum_j \frac{1}{2} dx_j\wedge d\hat{x}_j).	
		\eee
	This is different from the twisted double torus of $(T,\omega, B)$ by a B-field twist. See Chapter 6 for more discussion.
\end{remark}

\begin{remark}\label{AS2}
	The twisted doubling torus still comes with a complex structure $J$ which is twisted by the B-field. In matrix notation,
	
	\bee	
	 J=\left( \begin{array}{cc}
	 \omega^{-1}B & \omega^{-1} \\
	 -\omega-B\omega^{-1}B & -B\omega^{-1}
	\end{array}\right)
	=\left( \begin{array}{cc}
	 1 & 0 \\
	 -B & 1
	\end{array}\right)
	\left( \begin{array}{cc}
	 0 & \omega^{-1} \\
	 -\omega & 0
	\end{array}\right)
	\left( \begin{array}{cc}
	 1 & 0 \\
	 B & 1
	\end{array}\right).
	\eee
	
\end{remark}

The lift procedures are the same as in the case without B-field, the definition is copied:
\begin{definition}
	The lift of a coisotropic (possibly Lagrangian) brane $(C,\,\nabla)$ is defined to be
	\bee
\{(x,\,\hat{x})\in T\times T^{\vee}|x\in C \text{ and }\ \langle\hat{x},\,\gamma_x\rangle =(-1)^{\xi(\gamma_x)}hol_\nabla(\gamma_x),\   \forall\  \gamma_x \in \pi_1(C,\,x) \},\quad \pi_T^*\nabla
    \eee
	where $\gamma_x$ is any linear circle passing through $x$ within $C$. And $\xi:\,H_1(C)\rightarrow \mathbb{Z}/2$ such that
	\bee \xi(\gamma+\gamma')-\xi(\gamma)-\xi(\gamma')=c_1(\nabla)(\gamma\wedge \gamma') \ mod \ 2
	\eee
\end{definition}
Similarly to the case without B-field, the lifts are Lagrangian and complex submanifolds of $T\times T^\vee $:

\begin{proposition}\label{LiftisLag2}
	The lift of a coisotropic (possibly Lagrangian) brane is a Lagrangian brane in the doubling torus. And it is also a $J$-complex (see remark \ref{AS2}) submanifold of $T\times T^\vee$.
\end{proposition}
\begin{proof}
Note that
$$
\left(\begin{array}{cc}
	 \omega+B\omega^{-1}B & B\omega^{-1}\\
		 -\omega^{-1}B & -\omega^{-1}
		\end{array}\right)
=\left(\begin{array}{cc}
	1 & -B \\
	0 & 1
	\end{array}\right)
\left(\begin{array}{cc}
	\omega & 0 \\
	0 & -\omega^{-1}
	\end{array}\right)
\left(\begin{array}{cc}
	1 & 0 \\
	B & 1
	\end{array}\right).
$$
The proof goes over the same as for Proposition \ref{LiftisLag1} if we replace $F$ by $B+F$.
\end{proof}

\begin{remark}
Recall that a space filling coisotropic brane $(C,\,\nabla)$ comes with a complex structure $\omega^{-1}(F+B|_C)$. This complex structure coincides the complex structure \ref{AS2} on the lift $\boldsymbol{C}$ under the isomorphism $\Pi_{T}:\,\boldsymbol{C} \rightarrow C$, where $\Pi_{T}:\,T\times T^\vee \rightarrow T$ is the projection to the $T$ factor.
\end{remark}

\subsection{Mirror of the twisted double torus}

We recall that the mirror of a symplectic torus with B-field $(T,\,\omega,\,B)$ equipped with a SYZ fibration $F\rightarrow T \rightarrow Q$ is as follows:
\begin{equation}
H^1(F;\,\mathbb{C})/H^1(F;\,\mathbb{Z})+(B+i\omega) H_1(Q;\,\mathbb{Z}).
\end{equation}

\begin{example}\label{Tndouble}
	Let $(T,\,B+i\omega)=((\mathbb{R}/\mathbb{Z})^{n}_r\times (\mathbb{R}/\mathbb{Z})^{n}_\theta ,\,\tau dr \wedge d\theta)$ with SYZ fibers $F=T_\theta \times \{r\}$. The dual torus is $(T^\vee,\,(B+i\omega)^{-1})=((\mathbb{R}/\mathbb{Z})^{n}_{\hat{\theta}} \times (\mathbb{R}/\mathbb{Z})_{\hat{r}}^{n} ,\,\tau^{-1} d\hat{\theta}\wedge d\hat{r})$ with dual SYZ fibers $F^\vee =T_{\hat{r}}^n\times\{\hat{\theta}\}$.
	The mirror complex torus for the dual torus is
	$$E_{-\tau^{-1}}=\mathbb{C}^n/(\mathbb{Z}^n-(\tau^{T})^{-1} (\mathbb{Z}^n)).$$
	The twisted double torus is
	\bee\ba
	\Big(T\times T^{\vee}=(\mathbb{R}/\mathbb{Z})^{2n}\times (\mathbb{R}/\mathbb{Z})^{2n},
\quad  \boldsymbol{B}+i\boldsymbol{\omega}&=
	\frac{i}{2}(Im\tau +Re\tau (Im\tau)^{-1}Re\tau)dr\wedge d\theta \\
	&-\frac{i}{2}\tau(Im\tau)^{-1}dr\wedge d\hat{r}\\
	&+\frac{i}{2}(Im\tau)^{-1}\tau d\hat{\theta}\wedge d\theta \\
	&-\frac{i}{2} (Im\tau)^{-1} d\hat{\theta}\wedge d\hat{r}\Big).
	\ea\eee
	with SYZ fibers
	$\boldsymbol{F}=\{r\}\times T_\theta^n\times\{\hat{\theta}\}\times T_{\hat{r}}^n$.
	The mirror is
	
	\bee \mathbb{E}=\mathbb{C}^{4n}/(\mathbb{Z}^{2n}+\boldsymbol{\tau}^{T}(\mathbb{Z}^{2n})), \quad\boldsymbol{\tau}^{T} =\frac{i}{2}
	\left(\begin{array}{cc}
	Im\tau^T+Re\tau^T (Im\tau^T)^{-1}Re\tau^T & \tau^{T}(Im\tau^T)^{-1} \\
	-(Im\tau^{T})^{-1}\tau^T & -(Im\tau^T)^{-1}
	\end{array}\right).
	\eee
	Equipping SYZ fibers
	$\boldsymbol{F}=\{r\}\times T_\theta^n\times\{\hat{\theta}\}\times T_{\hat{r}}^n$
	with connection $\nabla=d+2\pi i (\phi d\theta +\kappa d\hat{r})$, local coordinates on the mirror are
	\bee\ba
	&\frac{1}{2\pi i}log(z_\frac{\partial}{\partial \theta})
	=\frac{i}{2}(Im\tau^T+Re\tau^T (Im\tau^T)^{-1}Re\tau^T)r
	+\frac{i}{2}\tau^{T}(Im\tau^T)^{-1}\hat{\theta}-\phi \\
	&\frac{1}{2\pi i}log(z_\frac{\partial}{\partial \hat{r}})
	=-\frac{i}{2}(Im\tau^{T})^{-1}\tau^T r
	-\frac{i}{2}(Im\tau^T)^{-1}\hat{\theta} -\kappa.
	\ea\eee
	Using alternative coordinates
	\bee\ba
		&u=\frac{1}{2\pi i}log(z_\frac{\partial}{\partial \theta})
		+\tau^T \frac{1}{2\pi i}log(z_\frac{\partial}{\partial \hat{r}})
		=\tau^T r-\tau^T \kappa -\phi \\
		&v=\frac{1}{2\pi i}log(z_\frac{\partial}{\partial \theta})
		+\bar{\tau}^T \frac{1}{2\pi i}log(z_\frac{\partial}{\partial \hat{r}})
		=-\hat{\theta}-\phi-\bar{\tau}^T\kappa,
	\ea\eee
	we have: $\boldsymbol{E}\simeq E_\tau \times E_{-\bar{\tau}}.$
	
\end{example}

\begin{remark}
	The mirror of the twisted double torus constructed above turns out to be isomorphic to the product of the original mirror with its complex conjugate $E\times \bar{E}$. This twisted double torus has the property that, even if a sheaf $\mathcal{E}\in Coh(E)$ corresponds to a coisotropic brane in $T$, a closely related sheaf on $E\times \bar{E}$ corresponds to a Lagrangian (which is the lift of the coisotropic brane) in $\boldsymbol{T}$.
\end{remark}
SYZ fibers $F\subset T$ lift to fibers $\boldsymbol{F}\subset \boldsymbol{T}$ which correspond to points in $E\times 0$, i.e. $v=0$.
In example \ref{Tndouble}, $(F=\{r\}\times T^n_{\theta},\ \nabla =d+2\pi i\phi d\theta) $ is lifted to
$(\boldsymbol{F}=\{r\}\times T^n_{\theta}\times\{\hat{\theta}=-\phi\}\times T^n_{\hat{r}},\ \nabla=d+2\pi i\phi d\theta)$ which corresponds to the point $u=\tau^Tr-\phi$, $v=0$.\\
Similarly, if a Lagrangian brane $L$ is mirror to a coherent sheaf $\mathcal{E}$, then the lift $\boldsymbol{L}$ is mirror to $\mathcal{E}\boxtimes \mathcal{E}_0$ on $E\times \bar{E}$, where $\mathcal{E}_0$ is a particular sheaf on $\bar{E}$.

\section{Floer Theory}

The Floer theory of $(T,B+i\omega)$ and its twisted double torus
$$(T\times T^{\vee},\quad \frac{1}{2}\sigma_0+\frac{i}{2}
\left(\begin{array}{cc}
\omega+B\omega^{-1}B & B\omega^{-1}\\
-\omega^{-1}B & -\omega^{-1}
\end{array}\right))$$
are deeply related to each other. In fact, the Floer cohomology between two Lagrangian brane $HF^*(L,\; L')$ in $(T,\; B+i\omega)$ can be identified with a subspace of the Floer cohomology of the lifts $HF^*(\boldsymbol{L},\; \boldsymbol{L}')$.

\subsection{Floer theory in 2-tori and their twisted doubles}
Let $(T=\mathbb{R}^2/\mathbb{Z}^2,\  B+i\omega=(b+ia)dr\wedge d\theta)$ equipped with SYZ fibration projecting to $r$ coordinate. \\
The coordinate on the mirror manifold is
$$z_{\theta}=e^{2\pi i\int{B+i\omega} }hol_\nabla(S^1_{\theta})
=e^{2\pi i\tau r}e^{-2\pi i\phi}=e^{2\pi i(\tau r-\phi)}.$$
Thus the mirror is the elliptic curve $E=\mathbb{C}/\mathbb{Z}+\tau \mathbb{Z}$, where $\tau=b+ia$. \\

\subsubsection{Theta functions on Elliptic curves}
Theta functions are holomorphic sections of holomorphic line bundles on elliptic curves. They can be constructed with the help of a holomorphic connection by periodizing a holomorphic section on the universal cover of the elliptic curve. If we begin with two gauge equivalent holomorphic connections on a degree $1$ line bundle, we get a priori different holomorhpic sections of the line bundle. The dimension of the space of holomorphic sections, which equals $1$, forces the two sections to be the same up to a constant. We can establish some magic formula for theta functions using this approach.
\begin{example}	
	Consider the degree $1$ holomorphic line bundle on the elliptic curve $E_\tau=\mathbb{C}/(\mathbb{Z}+\tau\mathbb{Z})$ where $\tau =b+ia$ with holomorphic connection $d+\frac{2\pi i}{a}ydx$. The transition function of the line bundle is given by
	             $$ s(z+1)=s(z),\quad s(z+n\tau)=e^{-\pi in^2b}e^{-2\pi inx}s(z).$$
	By starting with the holomorphic section $e^{-\frac{\pi}{a}y^2}$ on the universal cover of the torus and by periodizing it, we get a section of the line bundle given by
 	        $$s=\sum_n e^{-\frac{\pi}{a}(y+na)^2}e^{2\pi inx}e^{\pi in^2b}=\sum_n e^{\pi in^2\tau}e^{2\pi in(x+iy)}e^{-\frac{\pi}{a}y^2}.$$
\end{example}

\begin{example}
	Consider a gauge equivalent connection $d+\frac{\pi i}{a}(ydx-xdy)$. The transition function of the line bundle is given by
	$$ s(z+m+n\tau)=(-1)^{mn}e^{\frac{\pi i}{a}my}e^{\frac{\pi ib}{a}ny}e^{-\pi inx}s(z).$$
	By periodizing the holomorphic section $e^{-\frac{\pi}{2a}(x^2+y^2)}$ we get a holomorphic section
	\[
	\begin{aligned}
	&\sum_{m,n} (-1)^{mn}e^{\frac{\pi i}{a}my}e^{\frac{\pi ib}{a}ny}e^{-\pi inx}e^{-\frac{\pi}{2a}((x-m-nb)^2+(y-na)^2)}\\
	=&\sum_{m,n}(-1)^{mn}e^{-\frac{\pi}{2a}(m+n\tau)(m+n\bar{\tau})}e^{-\frac{\pi}{a}(m+n\bar{\tau})(x+iy)}e^{-\frac{\pi}{2a}(x^2+y^2)}.
	\end{aligned}
	\]
\end{example}

\begin{proposition}
	The holomorphic sections from the above examples are equal to each other up to a factor.
	      \be\label{thetaformula1}
	\begin{aligned}
	&\sum_{m,n}(-1)^{mn}e^{-\frac{\pi}{2a}(m+n\tau)(m+n\bar{\tau})}e^{-\frac{\pi}{a}(m+n\bar{\tau})z}e^{-\frac{\pi}{2a}z^2}\\
	=&\sum_{m,n}e^{-\frac{\pi}{2a}(z+m)^2}e^{-\frac{\pi}{a}n\bar{\tau}(z+m)}e^{-\frac{\pi}{2a}n^2\tau\bar{\tau}} \\
	=&\sqrt{2a}\sum_{l}e^{-\pi i\bar{\tau}l^2}\sum_{k}e^{\pi i\tau k^2}e^{2\pi ikz}.
	\end{aligned}
	       \ee
	
\end{proposition}
\begin{remark}
	The right hand side of equation \eqref{thetaformula1} is (up to a constant factor) the standard formula for the theta function for a degree $1$ line bundle, while the left hand side natrually arises in the Floer products on the twisted double torus. And this formula is the key to relate Floer theory of $(T,\; B+i\omega)$ and Floer theory of its twisted doubling torus. We will provide an aternative proof of a generalization of this formula \eqref{thetaformula2} using Fourier series.
\end{remark}

\subsubsection{Floer products on the 2-torus and on its twisted doubling torus}
\begin{example}

	Let $T=\mathbb{R}^{2}/\mathbb{Z}^{2},\; B+i\omega=\tau dr\wedge d\theta$ be a symplectic two torus with coordinates $r,\theta$. Let $L_0$ be a horizontal Lagrangian $\{\theta=0\}$ (mirror to the structure sheaf), $L_1$ be a slope $-1$ Lagrangian $\{\theta=-r\}$ (mirror to a degree 1 line bundle $\mathcal{L}^1$), $L_z$ be a vertical Lagrangian with position $r$ and a connection with holonomy $e^{-2\pi i\phi}$, $(L_z=\{r\}\times S^1_\theta,\;  \nabla=d+2\pi i\phi d\theta)$, where $z=\tau r-\phi$.
	The generator $s=(0,\;0)\in L_0\cap L_1$ of
	$HF(L_0,\; L_1)$ correspond to the $\vartheta$-function in $H^0(E,\; \mathcal{L}^1)$
		
		$$\vartheta =\sum_{n\in \mathbb{Z}} e^{\pi\tau i n^2}e^{2\pi in z}.$$

	The doubling torus is given by
	    \bee
	     (T\times T^{\vee},\; \frac{1}{2}(\sigma_0+i\Omega)=\frac{i}{2a}(\tau\bar{\tau}dr\wedge d\theta+\tau (d\hat{r}\wedge dr+d\hat{\theta}\wedge d\theta)+d\hat{r}\wedge d\hat{\theta}))
	     \eee
	     Complex coordinates on the mirror are given by
	     \bee\ba
	    z_{\frac{\partial}{\partial\theta}}&=e^{2\pi i \frac{i}{2a}(\tau\bar{\tau}r+\tau \hat{\theta})}e^{-2\pi i \phi};\\
	    z_{\frac{\partial}{\partial\hat{r}}}&=e^{2\pi i \frac{-i}{2a}(\tau r+\hat{\theta})}e^{-2\pi i \kappa};
	    \ea\eee
	    Note that
	    \bee\ba
	    z_{\frac{\partial}{\partial\theta}}(z_{\frac{\partial}{\partial\hat{r}}})^{\tau}&=e^{2\pi i(\tau r-\tau \kappa -\phi)};\\
	    z_{\frac{\partial}{\partial\theta}}(z_{\frac{\partial}{\partial\hat{r}}})^{\bar{\tau}}&=e^{2\pi i(-\hat{\theta}-\phi-\bar{\tau}\kappa)}.
	    \ea\eee
	
	Let
	$$u=\tau r-\tau \kappa -\phi,\quad v=-\hat{\theta}-\phi -\bar{\tau}\kappa$$
	 be the new coordinates of the mirror manifold. We can see that the mirror manifold is isomorphic to $E_{\tau}\times E_{-\bar{\tau}}$.
	
	The lifts $\boldsymbol{L}_0=\{\theta=0,\; \hat{r}=0\}$ and $\boldsymbol{L}_1=\{\theta=-r,\; \hat{r}=\hat{\theta}\}$ of $L_0$ and $L_1$ to the twisted doubling torus intersect in one point, which corresponds to a section of a line bundle on the mirror manifold. Considering the intersection with SYZ fibers $(\boldsymbol{F}=\{r\}\times S_\theta^1\times S_{\hat{r}}^1\times\{\hat{\theta}\},\; \nabla=d+2\pi i \phi d\theta+2\pi i \kappa d\hat{r})$, the Flor product $CF(\boldsymbol{L}_0,\;\boldsymbol{L}_1)\otimes CF(\boldsymbol{L}_1,\;\boldsymbol{F})\rightarrow CF(\boldsymbol{L}_0,\;\boldsymbol{F})$ is given by the following expression, summing the contribution of holomorphic triangles of edge length $(n+r)$ in $T$ and $(m+\hat{\theta})$ in $T^\vee$:
	\bee\ba
s&=\sum_{m,n}e^{2\pi i (\frac{i}{4a}\tau\bar{\tau}(n+r)^2+\frac{i}{4a}(m+\hat{\theta})^2+\frac{i}{2a}\tau(n+r)(m+\hat{\theta}))}e^{-2\pi i(n+r)\phi}e^{2\pi i(m+\hat{\theta})\kappa}\\
	&=\sum_{m,n}e^{-\frac{\pi}{2a}(n^2\tau\bar{\tau}+m^2+2\tau mn)}e^{-\frac{\pi}{a}(m+n\bar{\tau})u}e^{\frac{\pi}{a}(m+n\tau)v}e^{-\frac{\pi}{2a}(\tau\bar{\tau}r^2+\hat{\theta}^2+2\tau r\hat{\theta})}e^{-2\pi ir\phi}e^{2\pi i\hat{\theta}\kappa}
\ea\eee
	Claim:
	\begin{equation}\label{thetaformula2}
	\begin{aligned}
		&\sum_{m,n}e^{-\frac{\pi}{2a}(n^2\tau\bar{\tau}+m^2+2\tau mn)}e^{-\frac{\pi}{a}(m+n\bar{\tau})u}e^{\frac{\pi}{a}(m+n\tau)v}e^{-\frac{\pi}{2a}u^2}e^{\frac{\pi}{a}uv} e^{-\frac{\pi}{2a}v^2} \\
		=&\sqrt{2a}\sum_k e^{\pi i\tau k^2}e^{2\pi iku}\sum_l e^{-\pi i\bar{\tau}l^2}e^{2\pi ilv}
	\end{aligned}
	\end{equation}

	\begin{proof}
		
		Let $\displaystyle f(u)=\sum_m e^{-\frac{\pi}{2a}(u+m)^2}e^{-\frac{\pi}{a}(n\bar{\tau}-v)(u+m)}$. \\
		Then $f(u+1)=f(u)$, hence $f$ is equal to its Fourier series
		\bee
		\begin{aligned}
		f(u)&=\sum_k\int_{0}^{1}f(z)e^{-2\pi ikz}dz\ e^{2\pi iku} \\
		&=\sum_k \sum_m\int_{0}^{1}e^{-\frac{\pi}{2a}((z+m)^2+(2n\bar{\tau}-2v)(z+m))}e^{-2\pi ikz}dz\ e^{2\pi iku}  \\
		&=\sum_k e^{\frac{\pi}{2a}(n\bar{\tau}-v+2ika)^2}\sum_m\int_{0}^{1}e^{-\frac{\pi}{2a}((z+m)^2+(2n\bar{\tau}-2v+4ika)(z+m)+(n\bar{\tau}-v+2ika)^2)}dz\ e^{2\pi ik u} \\
		&= \sum_k e^{\frac{\pi}{2a}(n\bar{\tau}-v+2ika)^2}\int_{-\infty}^{\infty}e^{-\frac{\pi}{2a}((z+n\bar{\tau}-v+2ika)^2)}dz\ e^{2\pi ik u} \\
		&=\sum_k \sqrt{2a}\ e^{\frac{\pi}{2a}(n\bar{\tau}-v+2ika)^2} e^{2\pi ik u}. \\
		\end{aligned}
		\eee
		Then
		\bee
		\begin{aligned}
		LHS &=\sum_{m,n}e^{-\frac{\pi}{2a}(n^2\tau\bar{\tau}+m^2+2\tau mn)}e^{-\frac{\pi}{a}(m+n\bar{\tau})u}e^{\frac{\pi}{a}(m+n\tau)v}e^{-\frac{\pi}{2a}u^2}e^{\frac{\pi}{a}uv} \\
		&=\sum_n f(u)e^{-\frac{\pi}{2a}\tau\bar{\tau}n^2}e^{\frac{\pi}{a}n\tau v} \\
		&=\sum_n \sum_k \sqrt{2a}\ e^{\frac{\pi}{2a}(n\bar{\tau}-v+2ika)^2} e^{2\pi ik u} e^{-\frac{\pi}{2a}\tau\bar{\tau}n^2}e^{\frac{\pi}{a}n\tau v} \\
		&=\sum_k \sum_n \sqrt{2a}\ e^{-\pi i\bar{\tau}(n-k)^2}e^{\pi i\tau k^2}e^{2\pi i(n-k)v}e^{2\pi iku} \\
		&=\sqrt{2a} \sum_k e^{\pi i\tau k^2}e^{2\pi iku}\sum_l e^{-\pi i\bar{\tau}l^2}e^{2\pi ilv}
		\end{aligned}
		\eee
		
	\end{proof}
	
\end{example}

  This calculation shows that, up to a suitable rescaling (due to the discrepancy between the trivialization given by Floer complex and the usual holomorphic trivialization), the generator of $CF(\boldsymbol{L}_0,\,\boldsymbol{L}_1)$ corresponds to the section $\theta_\tau(u)\theta_{-\bar{\tau}}(v)$ of $\mathcal{L}\boxtimes\mathcal{L}$ on $E_\tau\times E_{-\bar{\tau}}$.
  
  In general, let $L_0$ be a horizontal Lagrangian $\{\theta=0\}$ (mirror to the structure sheaf), $L_d$ be a slope $-d$ Lagrangian $\{\theta=-dr\}$ (mirror to the degree d line bundle $\mathcal{L}^d$), $L_z$ be a vertical Lagrangian with position $r$ and a connection $\nabla=d+2\pi i\phi d\theta$, where $z=\tau r-\phi$.
The generators $s_k=(\frac{k}{d},\; 0)\in L_0\cap L_d$ of
$HF(L_0,\; L_d)$ correspond to the $\vartheta$-basis of $H^0(E,\; \mathcal{L}^d)$

$$\vartheta_{k/d} =\sum_{n\in \mathbb{Z}} e^{\pi\tau id (n+\frac{k}{d})^2}e^{2\pi id(n+\frac{k}{d}) z}.$$

On the doubling torus
$(T\times T^\vee,\;  \frac{i}{2a}(\tau\bar{\tau}dr\wedge d\theta+\tau (d\hat{r}\wedge dr+d\hat{\theta}\wedge d\theta)+d\hat{r}\wedge d\hat{\theta}))$, we get lifts of the above Lagrangian branes:
\bee\ba
\boldsymbol{L}_0&=\{\theta=0,\  \hat{r}=0\}\\
\boldsymbol{L}_d&=\{\theta=-dr,\  \hat{r}=d\hat{\theta}\}\\
\boldsymbol{L}_z&=(\{r\}\times S^1_\theta\times S^1_{\hat{r}}\times\{ \hat{\theta}=-\phi\},\
\nabla=d+2\pi i\phi d\theta).
\ea\eee

An argument similar to that given above for the case $d=1$ shows that, in $T\times T^\vee$, the generator $s_j\otimes \hat{s}_k \in CF(\boldsymbol{L}_0,\; \boldsymbol{L}_d)$ given by the point of $\boldsymbol{L}_0\cap\boldsymbol{L}_d$ with coordinate $r=j/d$ and $\hat{\theta}=k/d$ corresponds to
\bee
\sum_{l\in \mathbb{Z}/d}e^{2\pi ikl/d}\vartheta_{(j-l)/d}(u)\vartheta_{l/d}(v)\in H^0(E_\tau\times E_{-\bar{\tau}},\ \mathcal{L}^{d}\boxtimes \mathcal{L}^{d}).
\eee

\subsection{The general case for $T^{2n}$}

Let $T=(\mathbb{R}/\mathbb{Z})^{2n}$, $ B+i\omega=\tau dr\wedge d\theta=\sum \tau_{jk}dr_j\wedge d\theta_k \ (\tau \in M_{n\times n}(\mathbb{C}))$, with Lagrangian fibers $F=\{r\}\times T_\theta ^n$ and base $Q$, the mirror complex torus is
$$E\cong H^1(F;\,\mathbb{C})/(H^1(F;\,\mathbb{Z})+(B+i\omega)H_1(Q;\,\mathbb{Z}))\cong \mathbb{C}^n/(\mathbb{Z}^n+\tau^{T} (\mathbb{Z}^n)).$$
We consider the following three Lagrangian branes in $(T,\,B+i\omega)$:
\begin{itemize}
	\item $L_0=\{\theta=0\}$ with trivial connection. $L_0$ is mirror to the structure sheaf on $E$.
	
	\item $L_z=(\{r\}\times T_\theta ^n$, $\nabla=d+2\pi i(\phi d\theta)$. $L_z$ is mirror to the skyscraper sheaf at $z=\tau^{T}r-\phi$ on $E$.
	
	\item $L_D=(\{\theta=-Dr\}$, $\nabla=d-\pi i(r^{T}(Re\tau D-D^{T}Re\tau^{T} )dr)$
 where $D\in GL(n,\;\mathbb{Z})$ such that $Im\tau D=D^{T}Im\tau^{T}>0$ and $Re\tau D-D^{T}Re\tau^{T}\in M_{n\times n}(\mathbb{Z})$. The transition of the line bundle is 
 \be s(r+m)=(-1)^{\xi(m)}e^{\pi i m^{T}(Re\tau D-D^{T}Re\tau^{T})r}s(r), \ee
  where $\xi:H_1(T;\mathbb{Z})\rightarrow \mathbb{Z}/2$ such that 
 \be
 \xi(m_1+m_2)=\xi(m_1)+\xi(m_2)+m_1^{T}(Re\tau D-D^{T}Re\tau^{T})m_2 \in \mathbb{Z}/2.
 \ee
 Note that $F=Re\tau D dr\wedge dr=-B|_{L_D}$ satisfying the B-field condition. $L_D$ is mirror to a line bundle $\mathcal{L}_D$ on $E$ with first Chern class 
 $ D^{T}dr \wedge (d\phi-Re\tau^{T}dr)$, or equivalently $\frac{i}{2}(Im\tau)^{-1}D^{T}dz\wedge d\bar{z}$, where $z=\tau^{T}r-\phi$.

\end{itemize}

The intersections of $L_0$ and $L_D$ have coordinates $(D^{-1}k,\;0)$, where $k\in \mathbb{Z}^n$. Hence
$HF(L_0,\;L_D)=CF(L_0,\; L_D)=span_{\mathbb{C}}\{s_{(D,\;k)}\}_{k\in \mathbb{Z}^n}$. If $k\equiv k'\; mod\, D\mathbb{Z}^n$, then $s_{(D,k)}$ and $s_{(D,k')}$ correspond to the same intersection point, and the corresponding generators of $CF(L_0,L_D)$ coincide up to a multiplicative factor (which arise from the holonomy of the connection on $L_D$, see \eqref{eq4.8}).

Denote $e_D=(r,\;-Dr)$ the generator of $HF(L_D,\;L_z)$ and $e_0=(r,\;0)$ the generator of $HF(L_0,\;L_z)$. We will calculate the Floer product $\mu_2:HF(L_D,\;L_z)\otimes HF(L_0,\;L_D)\rightarrow HF(L_0,\;L_D)$.\\
Note that 
$$hol(L_D)=e^{-\pi i \langle m-D^{-1}k,\;(Re\tau D-D^{T}Re\tau^{T})r\rangle }e^{\pi i\langle (D^{-1}k),\;(Re\tau D-D^{T}Re\tau^{T})m\rangle}(-1)^{\xi(m)}$$ on the line segment from $s_{(D,\;k)}$ to $e_D$ along the vector $(m+r-D^{-1}k,\, -D(m+r-D^{-1}k))$, $m\in\mathbb{Z}^n$. (The first term is obtained by integrating the connection form, the rest come from the transition functions of the line bundle over $L_D$.)
\bee\ba
e_D\circ s_{(D,\,k)}
&=\sum_{ \triangle\in M(e_D,\, S_{(D,k)},\, e_0)} e^{2\pi i\int_{\triangle}B+i\omega}hol(\partial \triangle) e_0 \\
&=\quad\sum_{m\in \mathbb{Z}^n}\quad  e^{\pi i(B+i\omega)(m+r-D^{-1}k,\,D(m+r-D^{-1}k))}e^{-2\pi i\langle D(m+r-D^{-1}k),\,\phi \rangle}hol(L_D)e_0 \\
&=\quad\sum_{m\in \mathbb{Z}^n}\quad (-1)^{\xi(m)}e^{\pi i\langle D^{-1}k,\,(Re\tau D-D^{T}Re\tau^{T})m\rangle}e^{\pi i\langle \tau D(m-D^{-1}k),\,m-D^{-1}k\rangle} e^{2\pi i \langle Dm-k,\,\tau^{T}r-\phi\rangle}\\
&\qquad\qquad\qquad e^{\pi i\langle \tau Dr,\,r\rangle}e^{-2\pi i\langle Dr,\, \phi \rangle}e_0
\ea\eee

Up to a rescaling, whole expression coincide with the theta functions
\bee
\vartheta_{D, \,k}(z)=\sum_{m\in \mathbb{Z}^n} (-1)^{\xi(m)}e^{\pi i\langle D^{-1}k, \,(Re\tau D-D^{T}Re\tau^{T})m\rangle}e^{\pi i\langle \tau D(m-D^{-1}k), \,m-D^{-1}k\rangle} e^{2\pi i \langle Dm-k, \,z\rangle}
\eee
corresponding to the sections of the line bundle $\mathcal{L}_D$, which satisfy
\bee
\vartheta_{D,\,k}(z+\tau^{T}h)=(-1)^{\xi(h)}e^{-\pi i\langle \tau Dh,\,h\rangle}e^{-2\pi i \langle Dh,\,z\rangle} \vartheta_{D,\,k}(z),\; \vartheta_{D,\,k}(z+h)=\vartheta_{D,\,k}(z).
\eee
\be\label{eq4.8}
\vartheta_{D,\,k+Ds}(z)=(-1)^{\xi(s)}e^{\pi i\langle (Re\tau D-D^{T}Re\tau^{T})s,\,D^{-1}k\rangle} \vartheta_{D,\,k}(z)
\ee

The twisted double torus is given by
\bee\ba
\Big(T\times T^{\vee}=(\mathbb{R}/\mathbb{Z})^{2n}\times (\mathbb{R}/\mathbb{Z})^{2n},\quad  \frac{1}{2}\sigma_0+\frac{i}{2}\Omega=&
\frac{i}{2}(Im\tau +Re\tau (Im\tau)^{-1}Re\tau)dr\wedge d\theta \\
&-\frac{i}{2}\tau(Im\tau)^{-1}dr\wedge d\hat{r}\\
&+\frac{i}{2}(Im\tau)^{-1}\tau d\hat{\theta}\wedge d\theta \\
&-\frac{i}{2} (Im\tau)^{-1} d\hat{\theta}\wedge d\hat{r}\Big)
\ea\eee

Recall that the mirror of the twisted double torus is
\bee
\boldsymbol{E}\cong \mathbb{C}^n/(\mathbb{Z}^n+\tau^{T}\mathbb{Z}^n)
\times \mathbb{C}^n/(\mathbb{Z}^n-\bar{\tau}^{T}\mathbb{Z}^n)
\eee
with holomorphic coordinates
\bee
u=\tau^{T}(r-\kappa)-\phi, \qquad
v=-\bar{\tau}^T\kappa-\hat{\theta}-\phi.
\eee

The lifts of the three Lagrangians above are
\begin{enumerate}
	\item $\boldsymbol{L}_0=\{\theta=0,\;\hat{r}=0\}$ with trivial connection. $\boldsymbol{L}_0$ is mirror to the structure sheaf.
	
	\item $\boldsymbol{L}_z=\{\{r\}\times T_{\theta}^n\times T_{\hat{r}}\times \{\hat{\theta}\} \}$.
	
	\item $\boldsymbol{L}_D=\Big(\{\theta=-Dr,\;  \hat{r}-D^{T}\hat{\theta}=-(Re\tau D-D^{T}Re\tau^{T})r \},\quad \nabla=d-\pi ir^{T}(Re\tau D-D^{T}Re\tau^{T} )dr\Big)$ .
	
\end{enumerate}

\bee
\boldsymbol{L}_0\cap \boldsymbol{L}_D
=\{(r=D^{-1}k,\,\theta=0,\,\hat{r}=0,\ \hat{\theta}
=(D^{T})^{-1}(Re\tau D-D^{T} Re\tau^{T})D^{-1}k+(D^{T})^{-1}l)\}
\eee

Denote $p=D^{-1}k,\; q=(D^{T})^{-1}(Re\tau D-D^{T} Re\tau^{T})D^{-1}k+(D^{T})^{-1}l.$\\
Let $s_{k,\,l}=(p,\,0,\,0,\,q) \in CF(\boldsymbol{L}_0,\,\boldsymbol{L}_D),\;
e_0=(r,\,0,\,0,\,\hat{\theta}) \in CF(\boldsymbol{L}_0,\;\boldsymbol{L}_z)$ and $$e_{D}=(r,\;-D(r-p),\;D^{T}(\hat{\theta}-q)-(Re\tau D-D^{T} Re\tau^{T})(r-p),\;\hat{\theta}) \in CF(\boldsymbol{L}_D,\;\boldsymbol{L}_z). $$
The coefficient of $e_0$ in $\mu^2(e_D,\, s_{k,l})$ is then
\bee\ba
s_{k, \,l}=
\sum_{m, \,n\in \mathbb{Z}^n}
&e^{-\frac{\pi }{2}
\langle Im\tau^{-1} D^{T}(\hat{\theta}+n-q),\,\hat{\theta}+n-q\rangle
-\frac{\pi}{2}
\langle\bar{\tau}Im\tau^{-1}D^{T}\tau^{T}(r+m-p),\,r+m-p\rangle
-\pi\langle Im\tau^{-1}D^{T}\tau^{T}(r+m-p),\,\hat{\theta}+n-q\rangle} \\
&e^{-2\pi i\langle D(r+m-p),\,\phi\rangle +2\pi i\langle D^{T}(\hat{\theta}+n-q)-(Re\tau D-D^{T}Re\tau^{T})(r+m-p),\,\kappa\rangle}\\
& (-1)^{\xi(m)}e^{-\pi i\langle m-p,\,(Re\tau D-D^{T}Re\tau^{T})r\rangle}e^{\pi i\langle p,\,(Re\tau D-D^{T}Re\tau^{T})m\rangle} \\
=\sum_{m,\,n\in \mathbb{Z}^n}
& (-1)^{\xi(m)}e^{-\frac{\pi }{2}\langle Im\tau^{-1}D^{T}(n-q),\,n-q\rangle-\frac{\pi }{2}\langle\bar{\tau}Im\tau^{-1}D^{T}\tau^{T}(m-p),\,m-p\rangle-\pi \langle Im\tau^{-1}D^{T}\tau^{T}(m-p),\,n-q\rangle} \\
&e^{\pi i \langle p,\,(Re\tau D-D^{T}Re\tau^{T})m\rangle} e^{-\pi\langle Im\tau^{-1}D^{T}(n-q+\bar{\tau}^{T}(m-p)),\,u\rangle}e^{\pi\langle Im\tau^{-1}D^{T}(n-q+\tau^{T}(m-p)),\,v\rangle} \\
&e^{-\frac{\pi }{2}\langle Im\tau^{-1}D^{T}\hat{\theta},\,\hat{\theta}\rangle-\frac{\pi }{2}\langle\bar{\tau}Im\tau^{-1}D^{T}\tau^{T}r,r\rangle-\pi \langle Im\tau^{-1}D^{T}\tau^{T}r,\,\hat{\theta}\rangle}e^{-2\pi i\langle Dr,\phi \rangle+2\pi i\langle D^{T}\hat{\theta}-(Re\tau D-D^{T}Re\tau^{T})r,\,\kappa\rangle} \\
=\sum_{m,\,n\in \mathbb{Z}^n}
&e^{-\frac{\pi }{2}\langle Im\tau^{-1}D^{T}(u+n-q),\,u+n-q\rangle}
e^{-\pi \langle Im\tau ^{-1}D^{T}(\bar{\tau}^{T}(m-p)-v),\,u+n-q\rangle}e^{-2\pi i\langle D(m-p),\,n-q\rangle } \\
& (-1)^{\xi(m)}e^{\pi i\langle p,\, (Re\tau D-D^{T}Re\tau^{T})m\rangle}
e^{-\frac{\pi}{2}\langle \bar{\tau}Im\tau^{-1}D^{T}\tau^T(m-p),\,m-p\rangle}
e^{\pi \langle Im\tau^{-1}D^{T}\tau^{T}(m-p),\,v\rangle} \\
&e^{\frac{\pi}{2}\langle Im\tau^{-1}D^{T}u,\,u\rangle}
e^{-\pi \langle Im\tau^{-1}D^{T}u,\,v\rangle }
e^{-\frac{\pi }{2}\langle Im\tau^{-1}D^{T}\hat{\theta},\,\hat{\theta}\rangle-\frac{\pi }{2}\langle\bar{\tau}Im\tau^{-1}D^{T}\tau^{T}r,\,r\rangle-\pi \langle Im\tau^{-1}D^{T}\tau^{T}r,\,\hat{\theta}\rangle}\\
&e^{-2\pi i\langle Dr,\phi \rangle+2\pi i\langle D^{T}\hat{\theta}-(Re\tau D-D^{T}Re\tau^{T})r,\,\kappa\rangle}
\ea\eee

Using the same Fourier series manipulation as in the proof of \eqref{thetaformula2}, this can be rewritten as

\bee\ba
s_{k, \,l}=\sum_{m,\,s\in \mathbb{Z}^n}
&\sqrt{\frac{2}{det(Im\tau^{-1}D^{T})}}
e^{\frac{\pi}{2}\langle Im\tau^{-1}D^{T}(\bar{\tau}^{T}(m-p)-v)+2is,\,\bar{\tau}^{T}(m-p)-v+2iIm\tau^{T}D^{-1}s \rangle}e^{2\pi i \langle s,u-q\rangle}\\
&(-1)^{\xi(m)}e^{\pi i\langle p,\, (Re\tau D-D^{T}Re\tau^{T})m\rangle}
e^{-\frac{\pi}{2}\langle \bar{\tau}Im\tau^{-1}D^{T}\tau^T(m-p),\,m-p\rangle}
e^{\pi \langle Im\tau^{-1}D^{T}\tau^{T}(m-p),\,v\rangle} \\
&e^{2\pi i\langle D(m-p),\,q\rangle}e^{\frac{\pi}{2}\langle Im\tau^{-1}D^{T}u,\,u\rangle}
e^{-\pi \langle Im\tau^{-1}D^{T}u,\,v\rangle }\\
&e^{-\frac{\pi }{2}\langle Im\tau^{-1}D^{T}\hat{\theta},\,\hat{\theta}\rangle-\frac{\pi }{2}\langle\bar{\tau}Im\tau^{-1}D^{T}\tau^{T}r,\,r\rangle-\pi \langle Im\tau^{-1}D^{T}\tau^{T}r,\,\hat{\theta}\rangle}e^{-2\pi i\langle Dr,\,\phi \rangle+2\pi i\langle D^{T}\hat{\theta}-(Re\tau D-D^{T}Re\tau^{T})r,\,\kappa\rangle} \\
\ea\eee
After a change of variables $s\mapsto Ds-t$ and rearranging, this becomes
\bee\ba
s_{k, \,l}=\sum_{t\in \mathbb{Z}^n/D\mathbb{Z}^n}
&\sqrt{\frac{2}{det(Im\ \tau^{-1}D^{T})}}
e^{-2\pi i\langle l,\,D^{-1}(k-t)\rangle}
e^{\pi i\langle (Re\tau D-D^{T}Re\tau^{T})D^{-1}k,\,D^{-1}t\rangle}
\\
&\sum_{s\in \mathbb{Z}^n}
(-1)^{\xi(s)}e^{\pi i\langle (Re\tau D-D^{T}Re\tau^{T})s,\,D^{-1}t\rangle}
e^{\pi i \langle \tau D(s-D^{-1}t),\,s-D^{-1}t\rangle}e^{2\pi i \langle Ds-t,u\rangle} \\
&\sum_{m\in \mathbb{Z}^n}
(-1)^{\xi(m)}e^{-\pi i \langle (Re\tau D-D^{T}Re\tau^{T})m,\,p-D^{-1}t\rangle}
e^{-\pi i \langle \bar{\tau}D(m-p+D^{-1}t),m-p+D^{-1}t\rangle}
e^{2\pi i\langle D(m-p+D^{-1}t),\,v\rangle} \\
&e^{\frac{\pi}{2}\langle Im\tau^{-1}D^{T}(u-v),\,(u-v)\rangle}\\
&e^{-\frac{\pi }{2}\langle Im\tau^{-1}D^{T}\hat{\theta},\,\hat{\theta}\rangle-\frac{\pi }{2}\langle\bar{\tau}Im\tau^{-1}D^{T}\tau^{T}r,\,r\rangle-\pi \langle Im\tau^{-1}D^{T}\tau^{T}r,\,\hat{\theta}\rangle}e^{-2\pi i\langle Dr,\,\phi \rangle+2\pi i\langle D^{T}\hat{\theta}-(Re\tau D-D^{T}Re\tau^{T})r,\,\kappa\rangle} \\
=\sum_{t\in \mathbb{Z}^n/D\mathbb{Z}^n}
&\sqrt{\frac{2}{det(Im\ \tau^{-1}D^{T})}}
e^{-2\pi i\langle l,\,D^{-1}(k-t)\rangle}
e^{\pi i\langle (Re\tau D-D^{T}Re\tau^{T})D^{-1}k,\,D^{-1}t\rangle}
\vartheta_{D,t}(u)\bar{\vartheta}_{D,\,k-t}(v)\\
&e^{\frac{\pi}{2}\langle Im\tau^{-1}D^{T}(u-v),\,(u-v)\rangle}\\
&e^{-\frac{\pi }{2}\langle Im\tau^{-1}D^{T}\hat{\theta},\,\hat{\theta}\rangle-\frac{\pi }{2}\langle\bar{\tau}Im\tau^{-1}D^{T}\tau^{T}r,\,r\rangle-\pi \langle Im\tau^{-1}D^{T}\tau^{T}r,\,\hat{\theta}\rangle}e^{-2\pi i\langle Dr,\,\phi \rangle+2\pi i\langle D^{T}\hat{\theta}-(Re\tau D-D^{T}Re\tau^{T})r,\,\kappa\rangle} \\
\ea\eee
The factors on the last two lines correspond to the difference between the Floer basis and the usual holomorphic trivialization of $\mathcal{L}^D\boxtimes\mathcal{L}_0^D$ on the mirror, and can be dropped.
This leads to the formula:

\be\label{usub}
\sum_{l\in \mathbb{Z}^n/D^{T}\mathbb{Z}^n}s_{k,l}
=\sqrt{2det(Im\ \tau D)}\vartheta_{D,k}(u)\bar{\vartheta}_{D,0}(v) \in H^0(E_\tau\times E_{-\bar{\tau}}; \mathcal{L}^{D}\boxtimes \mathcal{L}^{D}_0)
\ee
By evaluating  $v$ at $0$, we recover (up to a scaling factor) $\vartheta_{D,k}(u)$ up to a constant factor which corresponds to the Floer product on the original torus $T$.

\section{The ``u-part" Floer cohomology $HF_u(\boldsymbol{L},\,\boldsymbol{L}')$}

Given two Lagrangian branes $L$ and $L'$ in $(T,\,\omega)$ with their lifts $\boldsymbol{L}$ and $\boldsymbol{L}'$, and assuming that $\mathcal{L},\,\mathcal{L'}$ are the mirror sheaves of $L$ and $L'$,  we expect 
$$HF^*(\boldsymbol{L},\,\boldsymbol{L}')\cong Ext_{E}^*(\mathcal{L},\,\mathcal{L}')\boxtimes Ext_{\bar{E}}^*(\mathcal{L}_0,\,\mathcal{L}'_0).$$
The goal of this chapter is to define a subspace $HF_u(\boldsymbol{L},\,\boldsymbol{L}')$ of $HF^*(\boldsymbol{L},\,\boldsymbol{L}')$ which is isomorphic to $Ext_{E}^*(\mathcal{L},\,\mathcal{L}')$, and thus to $HF(L,L')$.
\subsection{The ``u-part" of $HF^*(\boldsymbol{L},\,\boldsymbol{L})$}
Recall that the twisted doubling torus $T\times T^\vee$ is equipped with a natural complex structure from Remark \ref{AS2},
\begin{equation}
J=\left( \begin{array}{cc}
\omega^{-1}B & \omega^{-1} \\
-\omega-B\omega^{-1}B & -B\omega^{-1}
\end{array}\right)
\end{equation}
which induces a complex structure on the cotangent bundle $T^*(T\times T^\vee)$, still denoted by $J$,
\begin{equation}
J=\left( \begin{array}{cc}
B\omega^{-1} & \omega+B\omega^{-1}B \\
-\omega^{-1}  & -\omega^{-1}B
\end{array}\right).
\end{equation}

Given a Lagrangian brane $(L,\,\nabla)$ in $(T,\,\omega)$ with its lift $(\boldsymbol{L},\,\pi_T^*\nabla)$, we can compare the first order deformation of the objects.
A first order deformation of $(L,\,\nabla)$ is described by $(v;\,\alpha)$, where $v$ is a normal vector of $L$, and $\alpha$ is a real $1$-form. $(v;\,\alpha)$ maps to 
$[\iota_v (\omega-iB)+i\alpha]\in H^1(L;\,\mathbb{C})\cong HF(L,L)$. 
The corresponding first order deformation of $(\boldsymbol{L},\,\pi_T^*\nabla)$ is given by 
$(v,\,-\tilde{\alpha};\,\alpha,\,0)$,where $\tilde{\alpha}$ is the image of $\alpha$ under the identification $T^*T\cong TT^\vee$. 
This maps to 
\begin{equation}
\iota_{v-\tilde{\alpha}}(\frac{1}{2}
\left(\begin{array}{cc}
\omega+B\omega^{-1}B & B\omega^{-1}\\
-\omega^{-1}B & -\omega^{-1}
\end{array}\right)
-i\sigma_0)+i\alpha=\frac{1}{2}(1+iJ)(\iota_v (\omega-iB)+i\alpha)
\in H^1(\boldsymbol{L};\,\mathbb{C}).
\end{equation}

Note that the full first order deformation space of $(\boldsymbol{L},\,\pi_T^*\nabla)$ coincides with $H^1(\boldsymbol{L};\, \mathbb{C})\cong HF(\boldsymbol{L},\boldsymbol{L})$, and the first order deformations coming from lifts are exactly the $(0,1)$ part of $H^1(\boldsymbol{L},\,\mathbb{C})\cong HF(\boldsymbol{L},\,\boldsymbol{L})$ with respect to $J$ mentioned above. Then 
\be
HF^*(L,\,L)\cong H^*(L;\,\mathbb{C})=\bigwedge H^1(L;\,\mathbb{C})\cong \bigwedge H^{0,\,1}_{J}(\boldsymbol{L})=H^{0,\,*}_{J}(\boldsymbol{L})
\ee.
\begin{definition}
	$HF(\boldsymbol{L},\,\boldsymbol{L})_u:=\, H^{0,\,*}_{J}(\boldsymbol{L})\subset H^*(\boldsymbol{L};\,\mathbb{C})=HF(\boldsymbol{L},\,\boldsymbol{L})$.
\end{definition}

\subsection{The case of transversal intersection $L\cap L'$}
As in Chapter 4, consider a pair of Lagrangian branes $L,\,L'$, which are mirror to a pair of line bundles $\mathcal{L},\,\mathcal{L'}$. Let $\boldsymbol{L},\,\boldsymbol{L}'$ be the lifts in the double torus. Suppose $\mathcal{L'}\otimes \mathcal{L}^{-1}$ is ample, then the subspace of $HF^*(\boldsymbol{L},\,\boldsymbol{L'})$ spanned by 
$$\displaystyle \sum_{l\in \mathbb{Z}^n/(D'-D)^{T}\mathbb{Z}^n}s_{k,\,l}=\sqrt{2det(Im\ \tau (D'-D))}\vartheta_{D'-D,k}(u)\bar{\vartheta}_{D'-D,0}(v)$$
is isomorphic to $HF^*(L,L')$ by evaluation at $v=0$. 
We define the ``u-part" cohomology for transversal intersections as follows.

\begin{definition}
    Suppose $L=\{\theta=-Dr+c\} $ and $ L'=\{\theta=-D'r+c'\}$ intersect transversally, assume $\boldsymbol{L}\cap \boldsymbol{L}'=\{s_{k,l}\}$, 
    where $k\in \mathbb{Z}^n/(D'-D)\mathbb{Z}^n,\,
    l \in \mathbb{Z}^n/(D'-D)^{T}\mathbb{Z}^n$. Then we define

\be
    HF_u(\boldsymbol{L},\boldsymbol{L}'):=
    span\Big\{\sum_{l\in \mathbb{Z}^n/(D'-D)^{T}\mathbb{Z}^n}s_{k,l}\Big\}_{k\in \mathbb{Z}^n/(D'-D)\mathbb{Z}^n}
    \subset HF^*(\boldsymbol{L},\boldsymbol{L}').
    \ee
\end{definition}
\begin{definition}
    Let 
    \be
    \begin{aligned}
    \Pi_T:  HF^*(\boldsymbol{L},\boldsymbol{L}') &\rightarrow 
    HF_u^*(\boldsymbol{L},\boldsymbol{L}') \\
     s_j\otimes s_h' &\mapsto s_j\otimes (\frac{1}{det(D'-D)}\sum_l s_l').
    \end{aligned}
    \ee
    Define the product structure to be
    \be
    \begin{aligned}
    \mu^2_u: HF^*_u(\boldsymbol{L}',\,\boldsymbol{L}")&\otimes HF^*_u(\boldsymbol{L},\,\boldsymbol{L}') \rightarrow HF_u(\boldsymbol{L},\boldsymbol{L}") \\
    x&\otimes y \longmapsto \Pi_T(\mu^2(x,y)),
    \end{aligned}
    \ee
    where $\mu^2$ is the usual Floer product in $HF^*_u(\boldsymbol{L},\,\boldsymbol{L}')$.
\end{definition}

With these definitions, we state the main theorem:
\begin{theorem}
	For a pair of Lagrangian branes $L,\,L'$ which are mirror to a pair of line bundles $\mathcal{L},\,\mathcal{L'}$, let $\boldsymbol{L},\,\boldsymbol{L}'$ be the lifts in the twisted doubling torus. Suppose $\mathcal{L'}\otimes \mathcal{L}^{-1}$ is ample. Then the ``u-part" Floer cohomology $HF_u^*(\boldsymbol{L},\boldsymbol{L}')$ is isomorphic to $HF^*(L,\,L')$. And for two such pairs $L,L'$ and $L',L''$, the following diagram commutes
	\be\label{uspscomm}
	\begin{tikzcd}
	{HF^*(L',\,L'')\otimes HF^*(L,\,L')} \arrow[rr] \arrow[d, "\cong"] &  & {HF^*(L,\,L")} \arrow[d, "\cong"] \\
	HF^*_u(\boldsymbol{L}',\,\boldsymbol{L}'')\otimes HF^*_u(\boldsymbol{L},\,\boldsymbol{L}') \arrow[rr]                                 &  & HF_u(\boldsymbol{L},\boldsymbol{L}'').
	\end{tikzcd}
	\ee
\end{theorem}

\begin{proof}
The isomorphism $HF_u^*(\boldsymbol{L},\boldsymbol{L}')\cong HF^*(L,\,L')$ is given by
\be
s_k \mapsto 
\sum_{l\in \mathbb{Z}^n/(D'-D)^{T}\mathbb{Z}^n}s_{k,l}/(\sqrt{2det(Im\ \tau (D'-D))}\bar{\vartheta}_{D'-D,0}(0))
\ee
Then the commutative diagram follows from \eqref{usub} and formulas for $\vartheta$-functions.
\end{proof}

\subsection{Proposal for general case}

In the general case when $L$ and $L'$ intersect non transversally, we expect there is still a ``u-part" subspace $HF_u(\boldsymbol{L},\boldsymbol{L}')$ isomorphic to $HF(L,L)$, and a similar commutative diagram holds.
In a simple case, assume $T,\,L,\,L'$ admit a simultaneous  decomposition $T=T_1\times T_2,\, L=L_1\times L_2,\, L'=L_1'\times L_2'$, such that $L_1,\,L_2$ (resp. $L_1',\,L_2'$) are Lagrangian submanifolds of $T_1$ (resp. $T_2$), and assume $L_1=L_1'$, while $L_2$ intersects $L_2'$ transversally.
Then the twisted doubling torus and lifts of $L,\,L'$ also admit product structures, and 
\be
\begin{aligned}
HF^*(L,L')& \cong  HF^*(L_1,L_1')\otimes HF(L_2,L_2') \\
&\cong HF_u^*(\boldsymbol{L}_1,\boldsymbol{L}_1)\otimes HF_u^*(\boldsymbol{L}_2,\boldsymbol{L}_2') \\
&\cong H^{0,*}(\boldsymbol{L}_1)\otimes HF_u(\boldsymbol{L}_2,\boldsymbol{L}_2')
\end{aligned}
\ee
Inspired by this, and combining the definitions of ``u-part" Floer cohomology in the previous sections, we make the following tentative definition.

\begin{definition}
    For a pair of Lagrangian branes $L$ and $L'$, and their lift $\boldsymbol{L}$ and $\boldsymbol{L}'$, let
\bee 
HF^*_u(\boldsymbol{L},\,\boldsymbol{L}') :=
G- \textbf{invariant part of } H^{0,*}(\boldsymbol{L}\cap \boldsymbol{L}')
\eee 
Where $G$ is the discrete group of translations in the direction of $T^\vee$ acting on $\boldsymbol{L}\cap \boldsymbol{L}'$.
\end{definition}

\begin{conjecture}
For a pair of Lagrangian branes $L,\,L'$  with their lifts $\boldsymbol{L},\,\boldsymbol{L}'$ in the twisted doubling torus, the "u-part" Floer cohomology $HF_u^*(\boldsymbol{L},\boldsymbol{L}')$ is isomorphic to $HF^*(L,\,L')$. And for two such pair $L,L'$ and $L',L"$, the following diagram commutes:
	\be\label{uspscomm}
	\begin{tikzcd}
	{HF^*(L',\,L")\otimes HF^*(L,\,L')} \arrow[rr] \arrow[d, "\cong"] &  & {HF^*(L,\,L")} \arrow[d, "\cong"] \\
	HF^*_u(\boldsymbol{L}',\,\boldsymbol{L}")\otimes HF^*_u(\boldsymbol{L},\,\boldsymbol{L}') \arrow[rr]                                 &  & HF_u(\boldsymbol{L},\boldsymbol{L}").
	\end{tikzcd}
	\ee
\end{conjecture}

\section{Equivalence of $T$ and $T^\vee$ with B-field Twist}

Recall that the dual torus of $(T=V/\Lambda,B+i\omega)$ is $(T^\vee=V^\vee /\Lambda^\vee,\ (B+i\omega)^{-1} )$ under assumption \ref{assump}. Explicitly, $$(B+i\omega)^{-1}=(\omega+B\omega^{-1}B)^{-1}B\omega^{-1}-i(\omega+B\omega^{-1}B)^{-1}.$$

\begin{example}
Let $T=\mathbb{R}^2/\mathbb{Z}^2$ with $B+i\omega=\tau dr\wedge d\theta$, $\tau=b+ia,\ a>0$, then its dual torus is $T^\vee=\mathbb{R}^2/\mathbb{Z}^2$ with $(B+i\omega)^{-1}=-\tau^{-1}d\hat{r}\wedge d\hat{\theta}$. They are non symplectomorphic tori for generic $\tau$. However, their mirror manifolds are isomorphic as complex manifold.
\bee 
\begin{tikzcd}
\mathbb{C}/\mathbb{Z}+\tau \mathbb{Z} &  & \mathbb{C}/\mathbb{Z}-\tau^{-1}\mathbb{Z}. \arrow[ll, "\times \tau"']
\end{tikzcd}
\eee
This implies that their Fukaya categories are equivalent.
\end{example}
The phenomenon that dual tori have equivalent Fukaya categories is not obvious without referring to Homological Mirror Symmetry. 
However, the twisted doubling tori of $(T,B+i\omega)$ and $(T^\vee,\ (B+i\omega)^{-1} )$ are the same up to a B-field twist.
Explicitly, their twisted doubling tori are
\begin{equation*}
	(T\times T^\vee,\quad  \Omega=\frac{1}{2}
	\left(\begin{array}{cc}
	 \omega+B\omega^{-1}B & B\omega^{-1}\\
		 -\omega^{-1}B & -\omega^{-1}
		\end{array}\right),
		\quad
		 \sigma_0=\sum_j \frac{1}{2} dx_j\wedge d\hat{x}_j)	
	\end{equation*}
	and
	\begin{equation*}
	(T\times T^\vee,\quad  \Omega=\frac{1}{2}
	\left(\begin{array}{cc}
	 \omega+B\omega^{-1}B & B\omega^{-1}\\
		 -\omega^{-1}B & -\omega^{-1}
		\end{array}\right),
		\quad
		 -\sigma_0=-\sum_j \frac{1}{2} dx_j\wedge d\hat{x}_j)	
	\end{equation*}
	The difference of B-field is an integral class in $H^2(T\times T^\vee; \mathbb{R})$. Let
	\be
	\nabla_0=d-2\pi i (rd\hat{r}+\theta d\hat{\theta})
	\ee
	be a $U(1)$ connection on $T\times T^\vee$ with curvature $2\sigma_0$. Then we have an equivalence of Fukaya categories of the two doubling tori by a B-twist:
	\be
	\ba 
	Fuk(T\times T^\vee,\Omega,\sigma_0) &\rightarrow Fuk(T\times T^\vee,\Omega,-\sigma_0) \\
	(L,\nabla) &\mapsto (L,\nabla\otimes \nabla_0|_L)
	\ea 
	\ee
\begin{conjecture}
	The Fukaya category of a torus $(T,B+i\omega)$ is equivalent to the Fukaya category of its dual $(T^\vee,(B+i\omega)^{-1})$.
\end{conjecture}

\begin{remark}
	 The isomorphism between morphism spaces in $(T,B+i\omega)$ and $(T^\vee,(B+i\omega)^{-1})$ is not directly given by the above equivalence of doubling tori. It also involves the projection map associated with $HF(\boldsymbol{C},\boldsymbol{C}')_u$. \\
\end{remark}

\bibliographystyle{amsplain}

\end{document}